\documentclass[12pt]{article}
\usepackage{latexsym}
\usepackage{amsmath,amsfonts,amssymb,mathrsfs,graphicx}
\usepackage{amsthm}
\usepackage{xcolor}
\usepackage{verbatim}% for comment

\title{Self-stabilizing processes \\
}
\author{K.J. Falconer \\
\small{{\em Mathematical Institute,
University of St~Andrews, North Haugh, St~Andrews,}} \\
\small{{\em Fife, KY16~9SS, Scotland }} \\
\small{ and } \\
J. L\'{e}vy V\'{e}hel \\
\small{{\em Case Law Analytics \& Inria,}} \\
\small{{\em Universit\'{e} Nantes, Laboratoire de Math\'{e}matiques 
Jean Leray}}\\
\small{{\em2 Rue de la Houssini\`{e}re - BP 92208 - F-44322 Nantes Cedex, France}}}

\newcommand{\E}{\mathbb{E}}
\renewcommand{\P}{\mathbb{P}}

\newcommand\ed{\stackrel{{\rm dist}}{=}}

\newcommand\tod{\stackrel{{\rm dist}}{\rightarrow}}

\newcommand{\be}{\begin{equation}} % begin equation
\newcommand{\ee}{\end{equation}} % end equation

\newcommand\X{{\sf X}}
\newcommand\Y{{\sf Y}}
\newcommand\Pt{{\sf P}}

 \newtheorem{theo}{Theorem}[section]
 \newtheorem{cor}[theo]{Corollary}
 \newtheorem{lem}[theo]{Lemma}
 \newtheorem{prop}[theo]{Proposition}

\begin{document}
\maketitle

\begin{abstract}
\noindent We construct `self-stabilizing' processes $\{Z(t), t\in [t_0,t_1)\}$. These are random processes which when `localized', that is scaled around $t$ to a fine limit, have the distribution of an $\alpha(Z(t))$-stable process, where $\alpha$ is some given function on $\mathbb{R}$. Thus the stability index at $t$ depends on the value of the process at $t$. Here we address the case where $\alpha: \mathbb{R} \to (0,1)$. We first construct deterministic functions which satisfy a kind of autoregressive property involving sums over a plane point set $\Pi$. Taking $\Pi$ to be a Poisson point process then defines a random pure jump process, which we show has the desired  localized distributions.
 \end{abstract}

\section{Introduction and background}
\setcounter{equation}{0}
\setcounter{theo}{0}
The irregularity of classes of functions or stochastic processes may be described by various parameters. A H\"{o}lder exponent is often used, for example the classical Weierstrass function  $f(t) = \sum_{n=1}^\infty \lambda^{-nh} \sin(\lambda^nt) \ (\lambda>2, h>0)$ has constant H\"{o}lder exponent $h$, as does, almost surely, index-$h$ fractional Brownian motion. For $\alpha$-stable processes, the stability index $\alpha$ controls the intensity of  jumps. For such functions and processes, these parameters may be set depending on the application in mind. Nevertheless, in many situations the irregularity may change with time, for example when modelling financial markets where the volatility can vary widely. Thus it is natural to construct functions and processes where the parameter, say $h(t)$, depends on time $t$, and close to each time $t$ the process behaves as if the parameter essentially equals $h(t)$. This can be formalised in terms of  the scaling limit of the process about $t$, a notion termed `localizability'. Examples  include the generalized Weierstrass function with variable H\"{o}lder exponent $h(t)$, see \cite{DLM},  multifractional Brownian motion with variable index $h(t)$, see  \cite{AL,Ben,LP},  multistable processes, with variable $\alpha(t)$, see  \cite{FL,FLL,XFL,LLA,LLL,LL}, and multistable subordinators and multifractional Poisson processes, see \cite{MR}.

It has been observed that, for certain phenomena, the local irregularity of measurements may be related to their amplitude \cite{BEL,ELP}. This relationship might be expressed by  a function $\varphi$ such that the irregularity exponent $h$ of a function or process $f$ at time $t$ depends on  $f(t)$ in a way determined by $\varphi$, that is near to time $t$ the irregularity parameter is close to $h\big(\varphi(f(t))\big)$. Such functions are termed {\it self-regulating}, and self-regulating versions of the Weierstrass function, of multifractional Brownian motion and of a midpoint displacement process  have been constructed in  \cite{BEL,ELP}. 

The aim of this paper is to construct jump processes of a self-regulating nature. We introduce `self-stabilizing' processes, that is variants on $\alpha$-stable processes where the stability index $\alpha$ around time $t$ depends on the value of the process at time  $t$. The construction utilizes the Poisson sum representation of $\alpha$-stable processes as a sum over a point set in the plane. We first recall some basic constructions of stable processes. 
\medskip

{\it Symmetric $\alpha$-stable L\'{e}vy motion} $\{L_\alpha (t), t\geq 0\} (0<\alpha \leq 2)$,  is the stochastic process  with stationary independent increments such that $L_\alpha(0) = 0$ almost surely, and 
$L_\alpha (t) - L_\alpha (s)$ has the distribution of $S_\alpha((t-s)^{1/\alpha},0,0)$,
where $S_\alpha(c, \beta,\mu)$ denotes a stable random variable with stability-index $\alpha$, with  scale parameter $c$, skewness parameter $\beta$, and shift $\mu$. We summarize the relevant features of such processes; a detailed account may be found in \cite{Ber,EM,Bk_Sam}.
The stable motion $L_\alpha $ has stationary increments and is $1/\alpha$-self-similar in the sense that $L_\alpha (c t)$ and $c^{1/\alpha}L_\alpha (t)$ have the same law. There is a version of  $L_\alpha $ such that its sample paths are c\`{a}dl\`{a}g, that is right continuous with left limits. Throughout the paper we write 
$$r^{\langle s\rangle} = {\rm sign}(r)|r|^s, \quad r\in \mathbb{R}, s\in \mathbb{R}.$$
Then symmetric $\alpha$-stable L\'{e}vy motion has  a representation
\begin{equation}\label{sum}
L_\alpha  (t) = C_\alpha 
\sum_{(\X,\Y) \in \Pi} 1_{(0,t]}(\X) \Y^{\langle-1/\alpha\rangle}
\end{equation}
where $C_\alpha$ is a normalising constant given by
\begin{equation}\label{calpha}
C_\alpha = \Big(\int_0^\infty u^{-\alpha} \sin u\, du \Big)^{-1/\alpha},
\end{equation}
and where $\Pi$  is a Poisson point process on $\mathbb{R}^+ \times\mathbb{R}$  with plane Lebesgue measure $\mathcal{L}^2$ as mean measure, so that for a Borel set $A \subset \mathbb{R}^+ \times\mathbb{R}$ the number of points of $\Pi$ in $A$ is a Poisson process with  parameter $\mathcal{L}^2(A)$, independently for disjoint sets $A$, see \cite[Section 3.12]{Bk_Sam}. The sum \eqref{sum} is almost surely absolutely convergent if $0<\alpha<1$, but if $1\leq \alpha<2$ then \eqref{sum}  must be taken as the almost surely convergent limit as $n\to \infty$ of the sum over $\{(\X,\Y)\in \Pi : |\Y|\leq n\}$.
\medskip

{\it Multistable L\'{e}vy motion} $\{M_{\alpha} (t), t\geq 0\}$ is a variant that allows the stability index $\alpha$ in  \eqref{sum}  to vary with $t$. 
This may be done in two distinct ways. Given a suitable 
$\alpha: \mathbb{R}^+ \to (0,2)$, we can either let
$$
M_{\alpha}(t) =  
\sum_{(\X,\Y) \in \Pi} 1_{(0,t]}(\X) C_{\alpha(\X)} \Y^{\langle-1/\alpha(\X)\rangle},
$$
or we can define 
\begin{equation}\label{mtilde}
\widetilde{M}_{\alpha}(t) = C_{\alpha(t)} 
\sum_{(\X,\Y) \in \Pi} 1_{(0,t]}(\X) \Y^{\langle-1/\alpha(t)\rangle}.
\end{equation}
Then $M$ is a Markov process but $\widetilde{M}$ is not, although they are both
semi-martingales. Properties of multistable L\'{e}vy motions have been
investigated in \cite{FL,FLL,KFLL,XFL,LLA,LLL,LL}. 

In particular, under certain conditions $M_{\alpha}$ and $\widetilde{M}_{\alpha}$
are {\it localizable} 
\cite{Fal5,Fal6}, in the sense that near $t$ they `look like' $\alpha(t)$-stable processes, that is for each $t>0$ and $u \in \mathbb{R}$,
$$\frac{M_{\alpha}(t+ru) -M_{\alpha}(t)}{r^{1/\alpha(t)}} \tod L_{\alpha(t)}(u)$$
as $r\searrow 0$, where convergence is in distribution with respect to the Skorohod metric and consequently  in finite dimensional distributions, with the
same holding for $\widetilde{M}_{\alpha}$. 

With multistable motion, the local stability parameter depends on the time $t$. Our aim in this paper is to construct a process where the local stability parameter at time $t$ depends instead on the value of the process at time $t$. Such
a process might be termed `self-stabilizing'. 

Thus we would like, for suitable 
$\alpha: \mathbb{R} \to (0,2)$, to construct a process $\{Z(t)\equiv Z_{\alpha} (t), t\in [t_0,t_1)\}$ that is localizable in the sense that for all $t\in [t_0,t_1)$ and $u>0$
\begin{equation*}
\frac{Z(t+ru) -Z(t)}{r^{1/\alpha(Z (t))}}\bigg|\,  \mathcal{F}_t \  \tod \ L^0_{\alpha(Z (t))}(u)
\end{equation*}
as $r\searrow 0$,
where convergence is in  finite dimensional distributions and in distribution, and where $\mathcal{F}_t$ indicates conditioning on the process up to time $t$. (For notational simplicity it
is easier to construct $Z_{\alpha}$ with  local form as the non-normalised $\alpha$-stable processes 
$L^0_\alpha = C_\alpha^{-1} L_\alpha$.)

We achieve this in the case of  $\alpha: \mathbb{R} \to (0,1)$  in Theorem \ref{thmrand} by using Poisson sums to show that there exists a unique process such that 
\begin{equation*}
Z  (t) = a_0 + \sum_{(\X,\Y) \in \Pi} 1_{(t_0,t]}(\X)\, \Y^{\langle-1/\alpha(Z  (\X_-))\rangle} \qquad (t_0\leq t <t_1).
\end{equation*}
and then showing that this process has the desired localizability property in Theorem \ref{rtlocx}. We  first obtain a corresponding identity in a deterministic setting in Section \ref{sec:deter} and then extend this to the random setting in Section \ref{sec:stoch}. The case of   $\alpha: \mathbb{R} \to (0,2)$, where the sums need not be absolutely convergent needs an alternative approach, and we address this in a sequel paper \cite{FLV3}.

\section{Deterministic jump functions defined by plane point sets}\label{sec:deter}
\setcounter{equation}{0}
\setcounter{theo}{0}

This section is entirely deterministic. Given a countable discrete point set $\Pi$ in the plane we will construct real valued functions $f$ on an interval  $[t_0, t_1)$ such that $f(t)$  `jumps' when $t=x$ 
for each $(x,y) \in \Pi$, the magnitude of the jump depending both on $y$ and on the value of $\lim_{t\nearrow x} f(t)$. Thus the  jump behaviour of $f$ depends on the values of $f$ itself.

For $t_0<t_1$ let $D[t_0, t_1)$ denote the c\`{a}dl\`{a}g functions on $[t_0, t_1)$, that is functions $f$ that are right continuous, so $\lim_{t\searrow u}f(t) =f(u)$ for all $u\in [t_0, t_1)$, and have left limits, so the limit $f(u_{-}) : =\lim_{t\nearrow u}f(t)$ exists for all $u\in (t_0, t_1]$; note in particular that we require the left limit to exist at $t_1$. The space  $D[t_0, t_1)$ is complete under the supremum norm $\|\cdot\|_\infty$. Our results are also valid,  by trivial extension, on $D[t_0, t_1]$, but a half-open interval is more natural and convenient  when working with c\`{a}dl\`{a}g functions.  

Throughout this section, fix $0<a <b<1 $. Let  $\alpha:\mathbb{R} \to [a,b]$ be continuously differentiable with bounded derivative and let $a_0 \in \mathbb{R}$.
Let $\Pi \subset (t_0,t_1)\times \mathbb{R}$ be a set of points such that 
\be\label{fin}
 \sum_{(x,y) \in \Pi} |y|^{-1/b'}< \infty
\ee
for some $b<b'<1$; this will ensure convergence in \eqref{abconv} below.

Our aim in this section is to show that, given $\alpha$ and $\Pi$, there is a unique $f\in D[t_0, t_1)$ satisfying
\begin{equation}
f(t) = a_0 +\sum_{(x,y) \in \Pi} 1_{(t_0,t]}(x) y^{\langle-1/\alpha(f(x_-))\rangle}. \label{identint}
\end{equation}
We will obtain  \eqref{identint} by two methods which provide different insights and lead to different properties of $f$. First we will use a method based on Banach's contraction mapping theorem, and then we will give a constructive proof where $f$ is approximated by sums over finite point sets.

We will often need the following estimates. By the mean value theorem,
$$ y^{\langle-1/\alpha(v)\rangle} - y^{\langle-1/\alpha(u)\rangle}
 = (v- u)
y^{\langle-1/\alpha(\xi)\rangle}\log |y|\frac{\alpha'(\xi)}{\alpha(\xi)^2}\quad (y,u,v \in \mathbb{R} ),$$
where $\xi \in (u,v)$. In particular, 
\be 
\big| y^{\langle-1/\alpha(v)\rangle} - y^{\langle-1/\alpha(u)\rangle}\big |
 \leq   M\  |v- u| |y|^{-1/(a,b)}\quad (y,u,v \in \mathbb{R} ),\label{ydif}
 \ee
where for convenience we write
\be
|y|^{-1/(a,b)} = \max\big\{ |y|^{-1/a}\big(1+\big|\log|y|\big|\big), |y|^{-1/b}\big(1+\big|\log|y|\big|\big)\big\};\label{ydif1}
\ee
and
\be
M = \sup_{\xi\in\mathbb{R}} \frac{|\alpha'(\xi)|}{\alpha(\xi)^2} \label{ydif1a}
\ee
From  \eqref{fin} and \eqref{ydif1}
\be
\sum_{(x,y) \in \Pi} |y|^{-1/(a,b)}< \infty. \label{abconv}
\ee

\begin{comment}
Thus
\begin{equation}
\big| y^{\langle-1/\alpha(v)\rangle} - y^{\langle-1/\alpha(u)\rangle}\big | \leq  c(y)|v- u||y|^{-1/b}, 
\end{equation}
 where
 \begin{equation}
c(y) :=c_1\big(1+|y|^{1/b - 1/a}\big)\frac{\sup_{z\in \mathbb{R}} |\widetilde{\alpha}(z)|}{a^2}.
\end{equation}
When we work with point sets $\Pi$ in the plane it is convenient to write 
\begin{equation}
c(\Pi) := c(y_0)\quad \mbox{ where } \quad y_0 = \min\{|y| : (x,y) \in \Pi\}.
 \end{equation}
 \end{comment}

\subsection{Contraction approach}
In this section we use Banach's contraction mapping theorem to show that \eqref{identint} has a unique solution.
Given $a_0, \alpha$ and $\Pi$ as above,  define an operator $K$ on $D[t_0,t_1)$ by
\be\label{op}
K(f)(t) = a_0 +\sum_{(x,y) \in \Pi} 1_{(t_0,t]}(x) y^{\langle-1/\alpha(f(x_-))\rangle} \qquad (t_0 \leq t <t_1),
\ee
where the sum is absolutely convergent by \eqref{fin}. We need to check that $K$ is indeed an operator on $D[t_0,t_1)$.

\begin{lem}
The operator $K$ maps $D[t_0,t_1)$ into itself.
\end{lem}
\begin{proof} 
The set function $\mu(A) := \sum_{(x,y) \in \Pi} 1_{A}(x) y^{\langle-1/\alpha(f(x_-))\rangle}$ defines an absolutely finite signed measure on $[t_0,t_1)$, so in particular the continuity properties hold for this measure.

Let  $t_0\leq t <t+h< t_1$. As $h\searrow0$,
$$
K(f)(t+h) - K(f)(t) = \sum_{(x,y) \in \Pi} 1_{(t,t+h]}(x) y^{\langle-1/\alpha(f(x_-))\rangle} \to 0$$
since $\bigcap_{h>0} (t,t+h] = \emptyset$, so $K(f)$ is right continuous at all $t\in [t_0,t_1)$.

Now let   $t_0\leq t-h<t \leq t_1$. As $h\searrow0$,
$$
K(f)(t-h) = K(f)(t)\ -  \! \! \sum_{(x,y) \in \Pi} 1_{(t-h,t]}(x) y^{\langle-1/\alpha(f(x_-))\rangle} \to K(f)(t) \ -  \! \! \! \! \! \sum_{\{(x,y) \in \Pi\, : \, x=t\}} y^{\langle-1/\alpha(f(x_-))\rangle}
$$
since $\bigcap_{h>0} (t-h,t] =\{t\}$, so $K(f)$ has a left limit at $t$.
\end{proof}
\medskip

We would like to use  that  $K$ is a contracting operator on  $D[t_0,t_1)$ and apply Banach's contraction theorem. However, $K$ is  contracting only if the value of $|y|$ is not too small at points $(x,y)\in\Pi$. We make this assumption in part (a) of the proof, then in part (b) we apply this to the intervals between such `bad' $x$ and incorporate the jumps at these points directly.

\begin{theo}\label{thmdet}
With $a_0, \alpha$ and $\Pi$ as above, there exists a unique $f \in D[t_0,t_1)$ such that 
\begin{equation}\label{ident}
f(t) = a_0 +\sum_{(x,y) \in \Pi} 1_{(t_0,t]}(x) y^{\langle-1/\alpha(f(x_-))\rangle} \qquad (t_0\leq t < t_1).
\end{equation}
In particular $f(t_0) =a_0$. Moreover, for each $s$ and $t$ with $t_0\leq s<t<t_1$, $f(t)$ is completely determined given $f(s)$ and the points of the set $\Pi \cap ((s,t] \times \mathbb{R})$.
\end{theo}

\begin{proof} 
Let  $N\geq 1$ be a number chosen so that
 \begin{equation}\label{cond}
 M \sum_{(x,y) \in \Pi, |y|\geq N} |y|^{-1/(a,b)}< \frac{1}{2}, 
  \end{equation}
where $M$ is given by \eqref{ydif1a}. (Note that $\frac{1}{2}$ will be a contraction constant and could be replaced by any $0<k<1$.) We split the proof into two parts.

(a) First we establish a function satisfying \eqref{ident} under the assumption that  
 $|y|>N\geq 1$ for all $(x,y) \in \Pi$. 
Let  $K : D[t_0,t_1) \to D[t_0,t_1)$ be as in \eqref{op}.
For  $f,g \in D[t_0,t_1)$, 
\begin{eqnarray*}
\big| K(f)(t)  - K(g)(t)\big| & =& \Big|  \sum_{(x,y) \in \Pi} 1_{(t_0,t]}(x) \big[ y^{\langle-1/\alpha(f(x_-))\rangle} - y^{\langle-1/\alpha(g(x_-))\rangle}\big]\Big|\\
& \leq&  M\sum_{(x,y) \in \Pi} 1_{(t_0,t]}(x)  \big|f(x_-)- g(x_-)\big|
|y|^{-1/(a,b)},
\end{eqnarray*}
using \eqref{ydif}. 
As $|y|>N$ for all $(x,y) \in \Pi$, together with  \eqref{cond} this implies
$$\| K(f) - K(g)\|_\infty \leq {\textstyle\frac{1}{2}} \| f - g\|_\infty.$$
Since $(D[t_0,t_1),\|\cdot\|_\infty)$ is complete, Banach's  contraction mapping theorem gives a unique $f \in D[t_0,t_1)$ satisfying $K(f)(t) = f(t) $  for $t_0\leq  t<t_1$, that is satisfying \eqref{ident}.

\medskip

\noindent (b) We now dispense with the requirement that   $|y|>N$ for all $(x,y) \in \Pi$. 
The set $\{x: (x,y) \in \Pi : |y| \leq N\}$  is finite, and we number these $x$ so that
$t_0< x_1\leq x_2 \leq \cdots \leq x_n< t_1$. We will apply part (a) inductively  on the intervals between successive $x_i$.

Part  (a) with $t_1$ replaced by $x_1$ gives a function $f \in D[t_0,x_1)$ satisfying \eqref{ident} for $t_0\leq t<x_1$, to start the induction. Assume inductively that there exists $f \in D[t_0,x_k)$ satisfying \eqref{ident} for $t_0\leq t<x_k$, where $1\leq k < n$; we extend  $f$ to $D[t_0,x_{k+1})$.
Since $f \in D[t_0,x_k)$ the limit $f({x_k}_{-}) =\lim_{t\nearrow x_k}  f(t)$ exists.
Define 
\be\label{addterm}
f(x_k) = f({x_k}_{-}) +  \sum_{(x,y)\in \Pi\,:\, x=x_k} {y}^{\langle-1/\alpha(f({x_k}_{-}))\rangle};
\ee
for `typical' sets $\Pi$ there will be a single term in this sum.  Note that $|y|>N$ for all $\{(x,y) \in \Pi\ : x_{k}<x< x_{k+1}\}$ and that \eqref{cond} remains valid with $\Pi$ replaced by  this subset. Thus we may apply part (a) with $\{t_0,t_1\}$ replaced by $\{x_k,x_{k+1}\}$ taking $a_0 = f({x_k})$, to get  $f \in D[x_{k},x_{k+1})$ such that for $x_{k}\leq t <x_{k+1}$
\begin{eqnarray*}
f(t) &=& f({x_k}) \ +\sum_{ (x,y) \in \Pi\, :\, x_{k}< x< x_{k+1}} 1_{(x_k,t]}(x) y^{\langle-1/\alpha(f(x_-))\rangle} \\
&=& \lim_{t'\nearrow x_k}\Big[a_0 \ + \sum_{ (x,y) \in \Pi\, :\, t_0 < x< x_{k}} 1_{(t_0,t']}(x) y^{\langle-1/\alpha(f(x_-))\rangle}\Big]\\
&& \qquad+\sum_{(x,y)\in \Pi\,:\, x=x_k} {y}^{\langle-1/\alpha(f({x_k}_{-}))\rangle}+\sum_{ (x,y) \in \Pi\, :\,  x_{k}< x< x_{k+1}} 1_{(x_k,t]}(x) y^{\langle-1/\alpha(f(x_-))\rangle}\\
&=& a_0 \ +\sum_{ (x,y) \in \Pi\, :\,  t_0< x< x_{k+1}} 1_{(t_0,t]}(x) y^{\langle-1/\alpha(f(x_-))\rangle},
\end{eqnarray*}
using the inductive hypothesis. This extends $f$ to $D[t_0,x_{k+1})$, completing the inductive step. Finally, a similar argument on the interval $[x_n,t_1)$ extends $f$ from $D[t_0,x_{n})$ to $D[t_0,t_1)$ so $f$ satisfies \eqref{ident}.

For uniqueness, note that by Case (a), $f$ is uniquely defined on $[t_0,x_1)$, and since $f \in D[t_0,t_1)$, the value of $\lim_{t\nearrow x_1}f(t)$ and thus of $f(x_1) =\lim_{t\nearrow x_1}f(t) +\sum_{(x,y)\in \Pi\,:\, x=x_1} {y}^{\langle-1/\alpha(f({x_1}_{-}))\rangle}$ is uniquely specified. In the same way, under the inductive assumption that $f$ is uniquely defined on $[t_0,x_k]$, applying part (a) to the interval $[x_k,x_{k+1})$ gives that the extension of $f$ to  $[t_0,x_{k+1}]$  is unique, as is the final extension to $[t_0,t_1)$. 

By applying the result of the theorem to the interval $[s,t) \subset [t_0, t_1)$ taking $a_0 =f(s)$, there is a unique $g$ on $[s,t)$, and thus on $[s,t]$, satisfying 
$$g(t') = f(s) +\sum_{(x,y) \in \Pi} 1_{(s,t']}(x) y^{\langle-1/\alpha(g(x_-))\rangle}  \quad  (t'  \in [s,t]) .$$
From \eqref{ident} 
$$f(t') = f(s) +\sum_{(x,y) \in \Pi} 1_{(s,t']}(x) y^{\langle-1/\alpha(f(x_-))\rangle}   \quad ( t'  \in [s,t] \subset [t_0, t_1)  ),$$
so, by uniqueness,  $g(t') = f(t')$ for $t'  \in [s,t]$, and we conclude that $f(t)$ is determined by $f(s)$ and  $\Pi \cap ((s,t] \times \mathbb{R})$.
\end{proof} 

\subsection{Constructive approach}

It is useful to be able to approximate $f$ satisfying \eqref{identint} by finite sums. A natural approach is to define a sequence of functions $f_n \, (n \in \mathbb{N})$ by restricting the sums to points with $|y| \leq n$. Thus we let
\begin{equation}\label{identfin}
f_n(t) = a_0 +\sum_{ (x,y) \in \Pi \, :\,  |y|\leq n} 1_{(t_0,t]}(x) y^{\langle-1/\alpha(f_n(x_-))\rangle} 
\end{equation}
for  $t_0\leq t <t_1$. Then $f_n \in D[t_0,t_1)$ is uniquely defined as a sum over a finite set of points and is piecewise constant, so we may  evaluate $f_n$  using a finite number of inductive steps. List
$\{x: (x,y) \in \Pi : |y|\leq n\}$ as  $t_0<x_1 < \cdots < x_K <t_1$, and, for convenience, write $x_{K+1} = t_1$. Thus, inductively,
\begin{equation}\label{fnind}
\begin{array}{llll}
f_n(t) &=& a_0& (t\in [t_0, x_1) )\\
f_n(t) &=& {\displaystyle f_n({x_k}_{-}) +  \sum_{(x,y)\in \Pi\,:\, x=x_k, |y|\leq n} {y}^{\langle-1/\alpha(f_n({x_k}_{-}))\rangle} }\quad
& (t\in [x_k, x_{k+1}))
\end{array};
\end{equation}
again for typical $\Pi$  the sum in \eqref{fnind} will normally have a single term.

Note that one of the difficulties with the function given by \eqref{identfin} is that  if, as $n$ increases, a new  point $(x,y)$  enters the sum then, for all  existing $(x',y')$ with $x'>x$ and smaller $|y'|$, the summands ${y'}^{\langle-1/\alpha(f({x'}_{-}))\rangle}$ will change,  leading to a change in $f_n(t)$ for  $t>x$ that is amplified as $t$ increases past larger $x$ with $(x,y)\in \Pi$.

With the indirect definition of $f$ in \eqref{identint} it is not immediately obvious that $\{f_n\}$ converges to $f$.
This is shown in the following theorem which also provides an alternative, constructive, way of obtaining $f$ as the uniform limit of the $f_n$.

\begin{theo}\label{finlim}
Let  $a_0, \alpha$ and $\Pi$ be as above  and  let  $f_n\in D[t_0,t_1)\ (n \in \mathbb{N})$ be given by  \eqref{identfin}. Then $\{f_n\}$ is a Cauchy sequence in $(D[t_0,t_1),\|\cdot\|_\infty)$. Moreover,  $f_n \to f$ in $\|\cdot\|_\infty$ where $f$ is the unique function in $D[t_0,t_1)$ satisfying \eqref{identint}.
\end{theo}

\begin{proof}
Let $m>n\geq 2$. Again we list the points
$$\{x: (x,y) \in \Pi : |y|\leq n\} = \{t_0<x_1 < \cdots < x_K <t_1\}.$$
(Note that there may be several points $(x,y) \in \Pi$ with equal values of $x$; if we exclude this exceptional situation then the proof becomes notationally simpler, with the sums in  \eqref{ck} and elsewhere reducing to single terms.)
For notational convenience we set $x_0 := t_0$ and $x_{K+1} := t_1$. Write
\begin{equation}\label{ck}
c_k\  = \ M \!\!\! \sum_{(x,y)\in \Pi\,:\, x=x_k, |y|\leq n} |y|^{-1/(a,b)}  \qquad (1\leq k \leq K)
\end{equation}
and
\begin{equation}\label{epk}
\epsilon_k = \sum_{  (x,y) \in\Pi\, : \, x_{k-1}\leq x<x_k,  |y|> n}  |y|^{-1/b},\qquad (1\leq k \leq K+1)
\end{equation}
We compare $f_n(t)$ and $f_m(t)$ for increasing values of $t$ to
show by induction on $k$ that for $x_0 \leq t <x_k$
\begin{equation}\label{indhyp}
|f_m(t) - f_n(t)| \leq (1+c_1)\cdots(1+c_{k-1})(\epsilon_1 + \cdots +\epsilon_{k-1}) + \epsilon_k.
\end{equation}
Firstly, for $t_0 \leq t <x_1$, 
\begin{eqnarray*}
|f_m(t)-f_n(t)| &=& \Big|a_0 + \sum_{ (x,y) \in \Pi\, : \, t_0<x<x_1, \, n<|y|\leq m}1_{(t_0,t]}(x) y^{\langle-1/\alpha(f_m(x_-))\rangle} - a_0\Big|\\
&\leq &\sum_{ (x,y) \in \Pi\, : \, t_0<x<x_1, \, n<|y|\leq m}  |y|^{-1/b}\  \leq\  \epsilon_1,
\end{eqnarray*}
from \eqref{epk}, noting that $|y| > n \geq 1$ in the sum. 

Now assume inductively that \eqref{indhyp} is true for all $t_0 \leq t <x_k$ for some $k \, (1\leq k\leq K)$. Then for 
$x_k \leq t <x_{k+1}$, taking account of the jumps of $f_n$ and $f_m$ at $x_k$ and the jumps of $f_m$ in the interval $(x_k, t)$, 
\begin{eqnarray*}
f_n(t) &=& f_n(x_{k-}) + \sum_{(x,y)\in \Pi:x=x_k, |y|\leq n} {y}^{\langle-1/\alpha(f_n({x_k}_{-}))\rangle}\\
f_m(t) &=& f_m(x_{k-}) +\sum_{(x,y)\in \Pi: x=x_k, |y|\leq n} {y}^{\langle-1/\alpha(f_m({x_k}_{-}))\rangle}
\ + \sum_{ (x,y) \in \Pi\, : \, x_k \leq x\leq t, \, n<|y|\leq m} y^{\langle-1/\alpha(f_m(x_-))\rangle}. 
\end{eqnarray*}
Hence, using \eqref{ydif}, \eqref{ck} and  \eqref{epk},
\begin{eqnarray*}
\big|f_m(t) - f_n(t)\big| 
&\leq& \big|f_m(x_{k -}) - f_n(x_{k -})\big|\\
&&+\ \Big| \sum_{(x,y)\in \Pi:x=x_k, |y|\leq n} {y}^{\langle-1/\alpha(f_n({x_k}_{-}))\rangle}-  \sum_{(x,y)\in \Pi:x=x_k, |y|\leq n} {y}^{\langle-1/\alpha(f_m({x_k}_{-}))\rangle}\Big| \\
&& +  \sum_{(x,y) \in \Pi\, : \, x_k \leq x\leq t, \, n<|y|\leq m} |y|^{-1/b} \\
&\leq& \big|f_m(x_{k -}) - f_n(x_{k -})\big|(1+ c_k) + \epsilon _{k+1},
\end{eqnarray*}
from which \eqref{indhyp} follows with $k$ replaced by $k+1$ using the inductive hypothesis. By induction  \eqref{indhyp} holds for $t_0 \leq t <t_1$, and in particular, 
\begin{eqnarray}
|f_m(t)-f_n(t)| & \leq &  \prod_{k=1}^K (1+c_k) \sum_{k= 1}^{K+1} \epsilon_k \nonumber\\ 
 &  \leq  & \prod_{(x,y) \in \Pi, |y|\leq n} \big(1+M\, |y|^{-1/(a,b)} \big) \sum_{(x,y) \in \Pi, n<|y|\leq m} |y|^{-1/b} \nonumber \\
&  \leq  &   \exp\Big(M\sum_{(x,y) \in \Pi, |y|\leq n} |y|^{-1/(a,b)}\big)\Big)\sum_{(x,y) \in \Pi, |y| >n} |y|^{-1/b}\,\label{cauchy}
 \end{eqnarray}
using  \eqref{ck}. Since both of the series in \eqref{cauchy} converge,  this can be made arbitrarily small by taking $n$ sufficiently large, so $f_n$ is a Cauchy sequence in $(D[t_0,t_1),\|\cdot\|_\infty)$.

Since $(D[t_0,t_1),\|\cdot\|_\infty)$ is complete, $f_n$ converges to some $f$ in this space. Write \eqref{identfin} as
$$f_n(t) = a_0 +\sum_{ (x,y) \in \Pi} 1_{(t_0,t]}(x) y^{\langle-1/\alpha(f_n(x_-))\rangle} \ 
-\sum_{ (x,y) \in \Pi \, :\,  |y|> n} 1_{(t_0,t]}(x) y^{\langle-1/\alpha(f_n(x_-))\rangle}$$
for each $t\in [t_0,t_1)$. Letting $n\to\infty$ the first sum converges to 
$\sum_{ (x,y) \in \Pi} 1_{(t_0,t]}(x) y^{\langle-1/\alpha(f(x_-))\rangle}$ by the dominated convergence theorem with the summands dominated by $1_{(t_0,t]}(x) |y|^{-1/b}$ over a countable union of atomic measures. The second term is dominated by \\
$\sum_{ (x,y) \in \Pi \, :\,  |y|> n} 1_{(t_0,t]}(x) |y|^{-1/b} \to 0$, so $f$  satisfies \eqref{ident}, and is the unique such function by Theorem \ref{thmdet}.
\end{proof} 

The rate of convergence of $\{f_n\}$ may be estimated in terms of the point set $\Pi$ and $\alpha$. 

\begin{cor}\label{estep}
In the setting of Theorem \ref{finlim}, $f_n \to f\in D[t_0,t_1)$  with 
\begin{eqnarray}
\|f_n-f\|_\infty &  \leq  & \prod_{(x,y) \in \Pi, |y|\leq n} \big(1+M\, |y|^{-1/(a,b)} \big) \sum_{(x,y) \in \Pi, : \,  |y|>n} |y|^{-1/b} \label{errest0}\\
&\leq &\exp\Big(M\sum_{(x,y) \in \Pi, |y|\leq n} |y|^{-1/(a,b)}\Big)\sum_{  (x,y) \in\Pi \, : \,  |y|>n}  |y|^{-1/b},\label{errest}
\end{eqnarray}
 where $M$ is as in \eqref{ydif1a}. These series are convergent, so taking $n$ large makes this norm difference small.
\end{cor}

\begin{proof} 
Letting $m \to \infty$ in \eqref{cauchy},  $f_m(t) \to f(t)$ in $\|\ \|_\infty$  to give these estimates.
\end{proof}

\subsection{Dependence on $\Pi$}
It is natural to ask how the function $f$ satisfying \eqref{identint} varies with $\Pi$. In fact it is far from continuous in any reasonable sense. To illustrate this let $(x,y), (x',y')\in \Pi $ with $x<x'$  and assume  that there are no $(x'',y'')\in \Pi $ with $x<x''<x'$, and also that   $\{y'': (x,y'')\in \Pi\} = y$ and $\{y'': (x',y'')\in \Pi\} = y'$  (these assumptions have little effect on this example). From \eqref{identfin},
\be\label{discont}
f(x') - f(x_-)\  =\  y^{\langle-1/\alpha(f(x_-))\rangle}
+ y'^{\langle-1/(\alpha(f(x_-) +  y^{\langle-1/\alpha(f(x_-))\rangle} ))\rangle}.
\ee
If  $x$ is increased so that $x'<x$ then
$$f(x) - f(x'_-)\  =\  y'^{\langle-1/\alpha(f(x'_-))\rangle}
+ y^{\langle-1/(\alpha(f(x'_-) +  y'^{\langle-1/\alpha(f(x'_-))\rangle} ))\rangle};$$
thus, if $y$ and $y'$ are different, the increment of $f$ due to the combined effect of the two jumps at $x$ and $x'$ can change discontinuously as $x$ increases through $x'$.
[This phenomenon may not be unreasonable for applications: for a financial example, the result of changing pounds to euros just before Britain voted to leave the European Union was somewhat different to changing currency just afterwards!] 
Despite this example, it turns out that this type of discontinuity can only occur at point sets $\Pi$ where 
$x=x'$ for distinct $(x,y), (x',y')\in \Pi$.

For considerations of continuity, the supremum norm on the the c\`{a}dl\`{a}g functions $D[t_0, t_1)$ is inappropriate, since a small change in $\Pi$ may shift a jump point of $f$ slightly but with the resulting function far from $f$ in the norm metric. Moreover, $D[t_0, t_1)$ is not separable under the supremum norm, leading to  topological and measure theoretic difficulties. Thus we now consider the weaker Skorohod metric which regards c\`{a}dl\`{a}g functions as close even if there are small shifts in the jump points. 
The {\it Skorohod metric} $\rho_S$  on $D[t_0,t_1)$ may be defined as
\begin{equation}\label{Skoro}
\rho_S(f,g) = \inf_{w \in \mathcal{W}} \max \{\|w-i\|_\infty, \|f- g\circ w\|_\infty\},
\end{equation}
where $\mathcal{W}$ is the class of all strictly increasing homeomorphisms on $[t_0,t_1)$ and $i$ is the identity, so that $w$ allows for variation in the position of the jump points,   see \cite{Pol} for details.

A natural metric on the point sets in $\mathbb{R}^2$ should regard two sets as close if their points $(x,y)$ with small $|y|$ are `close in pairs'  with little weight being given to points with large $|y|$. Let
$$\mathcal{P}^b = \Big\{\Pi \subset (t_0,t_1)\times \mathbb{R}:  \sum_{(x,y) \in \Pi} |y|^{-1/b}< \infty\Big\}.$$
Then for $\Pi_1,\Pi_2 \in \mathcal{P}^b$ define
\begin{equation}\label{pointmet}
d(\Pi_1,\Pi_2) = \sup_{g\in \mathcal{G}}\Big| \sum_{(x,y) \in \Pi_1} g(x,y)|y|^{-1/b} -  \sum_{(x,y) \in \Pi_2} g(x,y)|y|^{-1/b}\Big|,
\end{equation}
where the supremum is over the class $\mathcal{G}$ of Lipschitzian functions $g:(t_0,t_1)\times \mathbb{R}\to \mathbb{R}$ such that $\|g\|_\infty \leq 1$ and  ${\rm Lip}\, g \leq 1$, where ${\rm Lip}\, g$ denotes the Lipschitz constant of $g$. (Note that an alternative, perhaps more natural, way of expressing $d$ is 
$$d(\Pi_1,\Pi_2) =  \sup_{g\in \mathcal{G}}\Big| \int g d\mu_{\Pi_1} -  \int g d\mu_{\Pi_2}\Big|,$$
 where 
$\mu_\Pi$ is the measure on  $\Pi\in \mathcal{P}^b$ given by 
$\mu_\Pi = \sum_{(x,y) \in \Pi} |y|^{-1/b} \delta_{(x,y)}$ with  $ \delta_{(x,y)}$  the unit point mass at $(x,y)$.) It is easy to see that $d$ is defined and is a metric on $\mathcal{P}^b$.

Theorems \ref{thmdet} and \ref{finlim} and \eqref{ident} show that if $b<b'<1$ then there are well-defined maps 
\be\label{psi}
\begin{array}{lrl}
&\psi_n: \mathcal{P}^{b'} \to D[t_0,t_1), &\mbox{ where $\psi_n(\Pi)$ is the unique $f_n$ satisfying \eqref{identfin}}\\
&\psi: \mathcal{P}^{b'} \to D[t_0,t_1), &\mbox{ where $\psi(\Pi)$ is the unique $f$ satisfying \eqref{ident}}
\end{array}.
\ee
 We show that $\psi_n$ and $\psi$ are continuous with respect to the metrics at point sets $\Pi$ apart from at certain exceptional $\Pi$. We let $L_{\pm n}$ denote the pair of lines $y=\pm n$.

\begin{prop}\label{ctyprop}
Let $a_0$ and $\alpha:\mathbb{R} \to [a,b]\subset(0,1)$ be as above and let $b<b'<1$. Then:

(i) for each $n\in \mathbb{N}$, $\psi_n: (\mathcal{P}^{b'},d) \to (D[t_0,t_1), \rho_S)$
 is continuous at all $\Pi_0 \in \mathcal{P}^{b'}$ such that $\Pi_0\cap L_{\pm n} = \emptyset$ and $x\neq x'$ for all distinct $(x,y),(x',y')\in \Pi_0$ with $|x|,|x'| < n$;
 
(ii) $\psi: (\mathcal{P}^b,d) \to (D[t_0,t_1), \rho_S)$
 is continuous at all $\Pi_0 \in \mathcal{P}^{b'}$ such that  $x\neq x'$ for all distinct $(x,y),(x',y')\in \Pi_0$.
\end{prop}

\begin{proof} 
For a given $n\in \mathbb{N}$ let  $\Pi_0$ satisfy the conditions of (i);  we show that $\psi_n$ is continuous at $\Pi_0$.
Order  $\{x: (x,y) \in \Pi_0 : |y|\leq n\}$ as  $t_0<x_1 < \cdots < x_K <t_1$ and let $x_{K+1} = t_1$; by the assumption on $\Pi_0$ the $x_i$ are all distinct. Writing $f_n = \psi_n(\Pi_0)$, the inductive definition  in \eqref{fnind} simplifies to 
\begin{equation}\label{fnind2}
\begin{array}{llll}
f_n(t) &=& a_0& (t\in [t_0, x_1) )\\
f_n(t) &=&  f_n(x_{k-1}) +   {y_k}^{\langle-1/\alpha(f({x_{k-1}}))\rangle} \quad
& (t\in [x_k, x_{k+1}))
\end{array}.
\end{equation}
Given $\epsilon>0$, if $\Pi'$ is formed by the points $(x'_1,y'_1),\ldots, (x'_K,y'_K)$ where $\max\{|x_i - x'_i| ,|y_i-y'_i|\}$ is sufficiently small for all $i$, not least so that $t_0<x'_1 < \cdots < x'_K <t_1$ and $|y'_i|<n$, then replacing $(x_i,y_i)$ by $(x'_i,y'_i)$ in \eqref{fnind2} gives a function $f'_n= \psi_n(\Pi')$ such that $\rho_S(\psi_n(\Pi_0),\psi_n(\Pi')) = \rho_S(f_n,f'_n)<\epsilon$, since we are using the Skorohod metric and the jump points move only slightly. This situation pertains if $d(\Pi_0,\Pi')$ is sufficiently small, so $\psi_n$ is continuous at $\Pi_0$.

Now let $\Pi_0 \in \mathcal{P}^{b'}$ satisfy the conditions of (ii) and let $\epsilon>0$. 
Assume first that there are arbitrarily large $n$ such that the pair of lines $L_{\pm n}$ has empty intersection with $\Pi_0$. Then for such $n$, in \eqref{errest} the left-hand sum is bounded independently of $n$ in a neighbourhood of $\Pi_0$, and by taking $n$ large enough the right-hand sum may be made arbitrarily small uniformly in a neighbourhood of $\Pi_0$, so we may find arbitrarily large $n \in \mathbb{N}$ and $\delta_1 >0$ such that if $d(\Pi_0,\Pi)<\delta_1$ then    
\be\label{rhos}
\rho_S\big(\psi_n(\Pi),\psi(\Pi)\big)\ \leq\ \|\psi_n(\Pi)-\psi(\Pi)\|_\infty\ \leq\ {\textstyle \frac{1}{3}}\epsilon.
\ee
By part (i), there is $\delta_2 >0$ such that if $d(\Pi_0,\Pi)<\delta_2$ and $n$ is sufficiently large then 
$$\rho_S\big(\psi_n(\Pi_0),\psi_n(\Pi)\big)  \leq {\textstyle \frac{1}{3}}\epsilon.$$
Combining with \eqref{rhos},  $\rho_S\big(\psi(\Pi_0),\psi(\Pi)\big)\leq \epsilon$ if $d(\Pi_0,\Pi) < \min\{\delta_1,\delta_2\}$, so $\psi$ is continuous at $\Pi_0$.

In the exceptional case where $y=\pm n$ has non-empty intersection with $\Pi_0$ for all sufficiently large $n$, the same argument holds on  replacing $f_n$ by $f_{\widetilde{n}}$, where $\widetilde{n}$ is a real number close to $n$ such that $y=\pm\widetilde{n}$ does not intersect $\Pi_0$.
 \end{proof}

\subsection{Some variants}

{\it Weighted case} \ \ We remark that very similar arguments to those in Theorems \ref{thmdet} and \ref{finlim}, but with more awkward derivative expressions, give that if $w: [a,b] \to \mathbb{R}$ is a continuously differentiable `weight' function, then there exists a unique $f \in D[t_0,t_1)$ such that 
\begin{equation}\label{identwt}
f(t) = a_0 +\sum_{(x,y) \in \Pi} 1_{(t_0,t]}(x) w\big(\alpha(f(x_-))\big)  y^{\langle-1/\alpha(f(x_-))\rangle} 
\end{equation}
for  $t_0\leq t <t_1$.
\medskip
 
\noindent{\it Non-autonomous case}\ \ 
In applications, especially in finance, we would not expect that the height of
the jump at location $x$ to depend only on the value of the function just
before the jump. In other words, the exponent of $y$, in, for example, \eqref{identwt} 
would also depend on other factors. For instance, if the price of, say, 
Pendragon (which is part of the FTSE 250) jumps at time $t$, it is likely that the size of the jump will be determined by the value of this asset just before the
jump, but also by the time  $t$ and probably the value of the composite index FTSE250 at this time. It is thus useful to allow the $\alpha$ function to depend 
not only on $f(x_-)$, but also on $t$ and an auxiliary function $g$. 
The results above go through in this slightly more general case with
minimal modification. More precisely, let $g: [t_0,t_1) \to \mathbb{R}$ be 
a measurable function and $\alpha:[t_0,t_1) \times \mathbb{R}^2 \to (a,b)$ be
continuously differentiable with bounded derivative with respect to its second 
variable. We now define $f_n$ by
\begin{equation}\label{identfinnonaut}
f_n(t) = a_0 +\sum_{ (x,y) \in \Pi \, :\,  |y|\leq n} 1_{(t_0,t]}(x) y^{\langle-1/\alpha(t,f_n(x_-),g(t))\rangle} 
\end{equation}
for  $t_0\leq t <t_1$.
 
\begin{theo}\label{finlimnonaut}
Let  $a_0, g,\alpha$ and $\Pi$ be as above. Let  $f_n\in D[t_0,t_1)\ (n \in \mathbb{N})$ be
given by  \eqref{identfinnonaut}. Then $\{f_n\}$ is a Cauchy sequence 
in $(D[t_0,t_1),\|\cdot\|_\infty)$, the limit $f$ of which satisfies
\begin{equation}\label{identintnonaut}
f(t) = a_0 +\sum_{(x,y) \in \Pi} 1_{(t_0,t]}(x) y^{\langle-1/\alpha(t,f(x_-),g(t))\rangle}. 
\end{equation}
\end{theo}

\begin{proof}
The proof only requires trivial modifications to that of Theorem \ref{finlim}
and is omitted.
\end{proof}
There are many other ways of defining functions as sums over point sets which yield such functional identities. Another possibility would be to replace powers of $y$ in the sums by more general functions of  the form $\phi(y,f(x_-))$, subject to reasonable decay of $\phi$ and $\partial\phi(y,u)/\partial u$  as $|y| \to \infty$.

\section{Sums over random sets and self-stabilizing processes}\label{sec:stoch}
\setcounter{equation}{0}
In this section we take the point set $\Pi$ in Section \ref{sec:deter} to be a random set given by a Poisson point process in the plane. This leads to a random function on $D[t_0,t_1)$ which we show is right-localizable with the desired self-stabilizing property.

\subsection{Sums over random sets}\label{sec3.1}
The underlying probability space for our processes will be that of a Poisson point process $\Pi$ on the plane, which we can take to be defined by the requirement that $N(A)$, the number of points of $\Pi$ in every Borel set $A\subset (t_0 ,t_1)\times \mathbb{R}$ is a random variable. The Poisson point process is defined by a {\em mean measure} $\mu$, so that $N(A)$ has Poisson distribution with mean $\mu(A)$, independently for disjoint Borel sets $A$.
Here we will take the mean measure to be plane Lebesgue measure ${\mathcal L}^2$ restricted to   $(t_0 ,t_1)\times \mathbb{R}$. The sums of functions of the points in $\Pi$ that we consider are random variables and thus the Poisson point process defines a probability distribution on $D[t_0 ,t_1)$ in a natural way.
See \cite{King} for these and other details of Poisson processes.

The probability measure is transferred to $D[t_0 ,t_1)$ by the mapping $\psi$ of  \eqref{psi}.  Using the continuity properties of Proposition \ref{ctyprop} it can be shown that  the Borel sets of $D[t_0 ,t_1)$ are  measurable.

We derive a random version of Theorem \ref{finlim}. As in Section 2, we take $\alpha:\mathbb{R} \to [a,b]$ to be a continuously differentiable function with bounded derivative  with  $0<a <b<1 $ and take $a_0\in \mathbb{R}$.

\begin{theo}\label{thmrand}
Let $\Pi \subset (t_0 ,t_1)\times \mathbb{R}$ 
be a  Poisson point process with ${\mathcal L}^2$ as mean measure. 
Then there exists a Markov process  $Z  $ on $[t_0 ,t_1)$  such that, almost surely, the sample paths are in $D[t_0 ,t_1)$ with  $Z  (t_0 )=a_0$ and 
\begin{equation}\label{identrand}
Z  (t) = a_0 +\sum_{(\X,\Y) \in \Pi} 1_{(t_0 ,t]}(\X)\, \Y^{\langle-1/\alpha(Z  (\X_-))\rangle} \qquad (t_0 \leq t <t_1).
\end{equation}
Writing 
\begin{equation}\label{zedn}
Z_n(t) = a_0 +\sum_{(\X,\Y) \in \Pi:|\Y|\leq n} 1_{(t_0 ,t]}(\X)\, \Y^{\langle-1/\alpha(Z_n  (\X_-))\rangle} \qquad (t_0 \leq t <t_1)
\end{equation}
then almost surely, $\|Z_n-Z\|_\infty \to 0$ as $n\to \infty$.
\end{theo}

\begin{proof}
By standard properties of Poisson point processes,   $\Pi$ is almost surely a countable set of isolated points such that 
\be\label{finrand}
 \sum_{(\X,\Y) \in \Pi} 1_{(t_0 ,t_1]}(\X)|\Y|^{-1/b'}< \infty
\ee
for every $0<b'<1$. For each such realisation of  $\Pi$,  Theorem \ref{thmdet} gives  a unique $Z \in D[t_0 ,t_1)$ satisfying \eqref{identrand} and Theorem \ref{finlim} gives that $\|Z_n-Z\|_\infty \to 0$. 

Let  $t_0  \leq s <t < t_1$ and let $\mathcal{F}_t$ denote the $\sigma$-field underlying the restricted point process $\Pi \cap \big((t_0 ,t]\times \mathbb{R}\big)$, so that  $(\mathcal{F}_t, t_0  \leq t < t_1 )$ is a filtration with respect to the usual ordering and $Z(t)$ is adapted to this filtration.
By the final part of Theorem 2.2,
$$Z  (t) = Z  (s) + \sum_{(\X,\Y) \in \Pi} 1_{(s,t]}(\X)\, \Y^{\langle-1/\alpha(Z  (\X_-))\rangle} $$
with the right-hand sum independent of $\mathcal{F}_s$, so 
for  $A \subset  \mathbb{R}$ a Lebesgue measurable set, $\mathbb{P} \big(Z  (t) \in A\,|\, \mathcal{F}_s\big) = \mathbb{P} \big(Z  (t) \in A\, |\, Z  (s)\big)$, thus $Z  $ is a Markov process. 
\end{proof}

We note here that it is possible to define a related random process satisfying
$$\widetilde{Z}  (t) = \sum_{(\X,\Y) \in \Pi} 1_{(t_0 ,t]}(\X)\, 
\Y^{\langle-1/\alpha(\widetilde{Z}  (t_-))\rangle} \qquad (t_0 \leq t <t_1)$$
instead of \eqref{identrand}. Thus $\widetilde{Z}$ is the self-stabilizing version
of $\widetilde{M}_\alpha$ in \eqref{mtilde}. However, $\widetilde{Z}$ is not a Markov process and is not
even causal, {\it i.e.} it cannot be constructed progressively in time. It
is therefore not adapted to the modelling of time series. It could however 
prove a useful model for other data, such as natural terrains.

It is useful to estimate the speed of convergence of  $Z_n$ to $Z$ in Theorem \ref{thmrand}, for example for purposes of simulating these random functions. However, getting reasonable estimates for the rates of convergence in Theorem \ref{thmrand} is awkward since  the probability of $(\X,\Y) \in \Pi$ with $|\Y|$ very small is high enough to make expectation estimates diverge. Nevertheless, we can obtain some concrete convergence estimates if we modify the setting slightly by assuming $|\Y|\geq K$ for some $K>0$; in practice this is a realistic assumption in that it essentially just excludes the possibility of $Z$ having unboundedly large jumps.

We will need Campbell's theorem on expectations of sums over Poisson point sets.

\begin{theo}[Campbell's Theorem] \label{camp}
Let $\Pi$ be a Poisson process on $S\subset \mathbb{R}^n$ with mean measure $\mu$ and let $f:S \to \mathbb{R}$ be measurable. Then
$$\E\bigg(\sum_{\Pt \in \Pi}f(\Pt)\bigg) = \int_S f(u)d\mu(u)$$
provided this integral converges, and 
$$\E\bigg(\exp\sum_{\Pt \in \Pi}f(\Pt)\bigg) = \exp\int_S \big( \exp f(u)-1\big)d\mu(u)$$
provided $\int_S \min\{|f(u)|,1\} d\mu(u)  <\infty$.
\end{theo}
\begin{proof}
See \cite[Section 3.2]{King}.
\end{proof}

\begin{theo}\label{thmrand3}
Let  $K>0$ and  let $\Pi $ be a Poisson process on $ (t_0,t_1)\times (-\infty, -K] \cup [K,\infty)$  
 with  mean measure ${\mathcal L}^2$. 
 With $Z$  and $Z_n$ as in \eqref{identrand} and \eqref{zedn} for this restricted domain $\Pi$, $\E\big(\|Z_n -Z\|_\infty\big)\to 0$ as $n\to \infty$ with  
\be
\E\big(\|Z_n -Z\|_\infty\big) \leq  
\frac{2b(t_1-t_0)}{1-b}\exp\bigg(2M(t_1-t_0)\int_K^\infty  y^{-1/(a,b)} \, dy\bigg) 
 n^{-(1-b)/b}. 
\label{diffbound1}
\ee

\end{theo}

\begin{proof}
This follows by setting  $f_n = Z_n$ and $f=Z$ in \eqref{errest0} and  taking the expectation. Then, for $n>K$, using  the independence of the  Poisson point process $\Pi$ on  $(t_0,t_1)\times [-n, -K] \cup [K,n]$ and $(t_0,t_1)\times  (-\infty, -n) \cup (n,\infty)$, and Theorem \ref{camp},	
\begin{align}
\E\big(\|Z_n &-Z\|_\infty\big) 
\ \leq\ 
\E\Big( \prod_{(\X,\Y) \in \Pi: |\Y|\leq n } \big(1+M|\Y|^{-1/(a,b)}\big)\sum_{(\X,\Y)\in \Pi  : |\Y|> n} |\Y|^{-1/b}\Big)\nonumber\\
&=\ 
\E\Big( \prod_{(\X,\Y) \in \Pi: |\Y|\leq n } \big(1+M|\Y|^{-1/(a,b)}\big)\Big)\ 
\E\Big(\sum_{(\X,\Y)\in \Pi  : |\Y|> n} |\Y|^{-1/b}\Big)\nonumber\\
&=\ 
\E\Big( \exp \sum_{(\X,\Y) \in \Pi: |\Y|\leq n }\log \big(1+M|\Y|^{-1/(a,b)}\big)\Big)\ 
\E\Big(\sum_{(\X,\Y)\in \Pi  : |\Y|> n} |\Y|^{-1/b}\Big)\nonumber\\
&=\ 
\exp\bigg(2(t_1-t_0)\int_K^n  M y^{-1/(a,b)} dy\bigg) 
2(t_1-t_0) \int_n^\infty  y^{-1/b}  dy
 \label{campap}
 \end{align}
 and letting $n\to \infty$ in the first integral and evaluating the second integral gives \eqref{diffbound1}.
\end{proof}

\subsection{Local properties of random functions and self-stabilizing processes}

Not only are the sample paths of the process given by Theorem \ref{thmrand} right-continuous, but we can obtain a local H\"{o}lder-type continuity estimate. We will show this by comparison with the 
$\alpha$-stable subordinator $S_\alpha $, for constant $0<\alpha<1$, which may be expressed as an almost surely convergent sum over a plane Poisson point process with mean measure ${\mathcal L}^2$ as 
$$
S_\alpha (t) \ :=\ \sum_{(\X,\Y) \in \Pi} 1_{(0 ,t]}(\X)\, |\Y|^{-1/\alpha}.
$$
Then $S_\alpha $ is a self-similar process with stationary increments such that for all $0<\epsilon<1/\alpha$ there is almost surely a random constant $C<\infty$ such that
\be\label{sub}
S_\alpha (t) \   \leq \ C t^{(1/\alpha)\,-\,\epsilon}  \qquad (t\geq 0),
\ee
 see for example \cite[Section III.4]{Ber} or \cite{Tak}.

\begin{prop}\label{cty}
Let $Z$ be the random function given by Theorem \ref{thmrand}. Then, given $0<\epsilon<1/b$, for each $t\in [t_0,t_1)$ there exists almost surely a random $C>0$  such that   for all $0<h <t_1-t$,
\be\label{hol}
|Z(t+h) -Z(t)|  \leq C h^{1/\alpha(Z(t))\,-\,\epsilon}.
\ee
\end{prop}
\begin{proof}
By \eqref{sub}, using that the subordinator has stationary increments,  there is an almost surely finite $C_1$ such that for all $0\leq h\leq t_1-t$,
$$\sum_{(\X,\Y) \in \Pi} 1_{(t,t+h]}(\X)\, |\Y|^{-1/\alpha(Z (t))\,-\,\epsilon/2}\ \leq\ C_1 h^{1/\alpha(Z (t))\,-\,\epsilon}.$$
Since $Z$ is almost surely right-continuous at $t$ and $\alpha$ is continuous, there is, almost surely, a random $0<H_0< \max\{t_1-t,1\}$ such that if $0\leq h \leq H_0$ then $|1/\alpha(Z(t+h)) - 1/\alpha(Z(t))|<\epsilon/2$.
Thus from \eqref{identrand}, almost surely, if $0\leq h\leq H_0$ then
\begin{eqnarray*}
|Z(t+h) -Z(t)| &\leq & \sum_{(\X,\Y) \in \Pi} 1_{(t,t+h]}(\X)\, |\Y|^{-1/\alpha(Z (\X_-))}\\
&\leq & C_2\sum_{(\X,\Y) \in \Pi} 1_{(t,t+h]}(\X)\, |\Y|^{-1/\alpha(Z (t))\,-\,\epsilon/2} \\
&\leq &   C_1 C_2 h^{1/\alpha(Z(t))\,-\,\epsilon}
\end{eqnarray*}
where $C_2$ is a random constant. By increasing $C_1C_2$ to a suitable value $C$ we can ensure that  \eqref{hol} holds for all $0\leq h<t_1-t$.
\end{proof}

We next show that near a time $t$, the random function $Z$ `looks like' an $\alpha$-stable process. Recall that a process $W$ is {\it localizable} at $t$ with a process $Y$ as its {\it local form}  if 
\begin{equation}\label{loc}
\frac{W(t+ru) -W(t)}{r^{1/\alpha}} \  \to \ Y(u)\qquad (u\in I)
\end{equation}
as $r\searrow 0$, where $I$ is an interval containing $0$ and convergence is in finite dimensional distributions. 
We say $W$ is {\it strongly localizable} if the convergence is in distribution with respect to an appropriate metric on the function space, see for example \cite{Fal6, FL}.
In our case, with the nature of $Z$ near $t$ depending on  $Z(t)$, it only makes sense to consider limits in \eqref{loc} for $u\geq 0$, in which case we refer to the process as {\it right-localizable}.

We will show that the process $Z$ is right-localizable at each $t$ with local form  an $\alpha(Z(t))$-stable process, so that $Z$ may indeed be thought of as self-stablizing. We write $L^0_\alpha $ for the non-normalized $\alpha$-stable process, which has a representation
\begin{equation}\label{sumnnx}
L^0_\alpha  (t) =\sum_{(\X,\Y) \in \Pi} 1_{(0,t]}(\X) \Y^{\langle-1/\alpha\rangle}\qquad (t\geq 0).
\end{equation}
As before $\mathcal{F}_t$ is the $\sigma$-field underlying the point process $\Pi \cap \big((t_0 ,t]\times \mathbb{R}\big)$.

\begin{theo}\label{rtlocx}
Let $Z$ be the process given by Theorem \ref{thmrand}. Then $Z$ is strongly right-localizable at each $t\in [t_0,t_1)$, in the sense that
\be\label{distlocx}
\frac{Z(t+ru) -Z(t)}{r^{1/\alpha(Z (t))}}\bigg|\,  \mathcal{F}_t \  \tod \ L^0_{\alpha(Z (t))}(u)
\ee
as $r\searrow 0$, where convergence is in distribution with respect to $(D[0,t_1), \rho_S)$, where $ \rho_S$ is the Skorohod metric, and so is also convergent in finite dimensional distributions.
\end{theo}
\begin{proof}
Let $t\in [0,t_1)$; throughout this proof we condition on $ \mathcal{F}_t $. Let $u \in [0,1]$ and let $0<r<t_1-t$. 
We  compare $Z(t+ru) -Z(t)$ and $L^0_{\alpha(Z (t))}(t+ru)-L^0_{\alpha(Z (t))}(t)$,  defined with respect to the same  Poisson point process $\Pi$ with mean measure $\mathcal{L}^2$. Let $0<\epsilon<1/(3b)$. Then for $u \in [0,1]$ and $0<r<t_1-t$,
\begin{align*}
\big|\big(Z(t+ & ru)  -Z(t)\big) -\big(L^0_{\alpha(Z (t))}(t+ru)-L^0_{\alpha(Z (t))}(t)\big)\big| \\
&=\ \Big| \sum_{(\X,\Y) \in \Pi} 1_{(t,t+ru]}(\X)\, \Y^{\langle-1/\alpha(Z  (\X_-))\rangle}  -   \sum_{(\X,\Y) \in \Pi} 1_{(t,t+ru]}(\X) \, \Y^{\langle-1/\alpha(Z(t))\rangle}  \Big| \\
&\leq\ \sum_{(\X,\Y) \in \Pi} 1_{(t,t+ru]}(\X)\, \big|\Y^{\langle-1/\alpha(Z  (\X_-))\rangle}- \Y^{\langle-1/\alpha(Z(t))\rangle}\big|\\
&\leq\ M\sum_{(\X,\Y) \in \Pi} 1_{(t,t+ru]}(\X)\,\big|Z (\X_-) - Z(t) \big|  |\Y|^{-1/(a,b) } \\
&\leq\ M C_1 r^{1/\alpha(Z(t))\,-\,\epsilon} \,  \sum_{(\X,\Y) \in \Pi} 1_{(t,t+ru]}(\X)\, |\Y|^{-1/b \, -\epsilon} \\
&\leq\ MC_1 r^{1/\alpha(Z(t))\,-\,\epsilon}C_2 (ru)^{1/b\,-\,2\epsilon}\\
& \leq C_3 r^{1/\alpha(Z(t))+1/b\,-\,3\epsilon}
\end{align*}
where $C_1,C_2,C_3$ are almost surely finite random constants. Here we have used \eqref{ydif}, inequality \eqref{hol}, and \eqref{sub} noting that the final sum is a stable subordinator. Thus, almost surely, there is a finite random constant $C_3$ such that  for $u\in [0,1]$   and $0<r<t_1-t$,
\begin{eqnarray*}
\bigg\| \frac{Z(t+  ru) - Z(t)}{r^{1/\alpha(Z(t))}} -\frac{L^0_{\alpha(Z (t))}(t+ru)-L^0_{\alpha(Z (t))}(t)}{ r^{1/\alpha(Z(t))}}\bigg\|_\infty 
&\leq&
C_3 \frac{r^{1/\alpha(Z(t))+1/b\,-\,3\epsilon}}{ r^{1/\alpha(Z(t))}}\\
&=& C_3  r^{1/b\, -3\epsilon} \to 0
\end{eqnarray*}
almost surely as $r\searrow 0$. In particular, as $\|\cdot\|_\infty$ dominates $\rho_S$ on $D[0,t_1)$,
$$ \rho_S \bigg( \frac{Z(t+  ru) - Z(t)}{r^{1/\alpha(Z(t))}}, \frac{L^0_{\alpha(Z (t))}(t+ru)-L^0_{\alpha(Z (t))}(t)}{ r^{1/\alpha(Z(t))}}    \bigg) \to 0$$
almost surely and in probability. Using scaling and stationary increments of the $\alpha(Z (t))$-stable process,
$$\frac{L^0_{\alpha(Z (t))}(t+ru)-L^0_{\alpha(Z (t))}(t)}{ r^{1/\alpha(Z(t))}} 
\ \ed \  L^0_{\alpha(Z (t))}(u)-L^0_{\alpha(Z (t))}(0) \ed L^0_{\alpha(Z (t))}(u),$$
so we conclude, using \cite[Theorem 3.1]{Bil} to combine convergence in probability and in distribution, that  
$$\frac{Z(t+  ru) - Z(t)}{r^{1/\alpha(Z(t))}}\bigg|\,  \mathcal{F}_t \  \tod \ L^0_{\alpha(Z (t))}(u)$$
 as $r\searrow 0$. Convergence in finite dimensional distributions is an immediate consequence.

\end{proof}

\subsection{Some variants}\label{sec3.2}

{\it Weighted case} \ \ As in the deterministic case, these results may be extended to include a weight function  $w: [a,b] \to \mathbb{R}^+$ that is bounded and continuously differentiable with bounded derivative. Thus a refinement of Theorem \ref{thmrand} gives a process $Z$ with sample paths in   $D[t_0,t_1)$ such that 
\begin{equation}\label{randwt}
Z (t) = \sum_{(\X,\Y) \in \Pi} 1_{(0,t]}(\X)\,w\big(\alpha(Z  (\X_-))\big)\,   \Y^{\langle-1/\alpha(Z  (\X_-))\rangle}. 
\end{equation}
In particular, taking $w(\alpha) = C_{\alpha}$ to be the normalizing constant for $\alpha$-stable L\'{e}vy motion \eqref{calpha}, there is a process satisfying
$$Z (t) = \sum_{(\X,\Y) \in \Pi} 1_{(0,t]}(\X)\,C_{\alpha(Z  (\X_-))}\,   \Y^{\langle-1/\alpha(Z  (\X_-))\rangle},$$
which, by the same arguments used to prove Theorem \ref{rtlocx}, satisfies
$$\frac{Z(t+  ru) - Z(t)}{r^{1/\alpha(Z(t))}}\bigg|\,  \mathcal{F}_t \  \tod \ L_{\alpha(Z (t))}(u)$$
as $r\searrow 0$, where $L_\alpha$ is standard (normalized) $\alpha$-stable L\'{e}vy motion. 
\medskip

\noindent{\it Tempered self-stabilizing processes} \ \ 
For certain applications, it is essential that the stochastic processes used
for modelling possess an expectation or even have finite variance. This
is particularly the case in financial engineering, where pricing typically
amounts to taking expectations. Because of the constraint $b<1$, the processes
defined in the previous sections do not meet these requirements.
Popular models in financial applications that retain some of the useful 
properties of stable processes but possess finite moments of all orders are
ones belonging to the class of  {\it tempered stable processes}
\cite{JR}. Well-known instances include the Variance-Gamma and the CGMY 
processes \cite{CGMY}.
A tempered version of self-stabilizing process may easily be constructed
using the theory developed above.

Let us briefly recall the shot noise representation of tempered stable processes
given in \cite{JR}. We need the following ingredients:
\begin{itemize}
\item $(\Y_i)_{i \geq 1}$ is a sequence of arrival times of a Poisson process 
with unit mean arrival time, 
\item $(\X_i)_{i \geq 1}$ is a sequence of i.i.d. random
variables with uniform distribution on $(0,t_1)$, 
\item $(\eta_i )_{i \geq 1}$ is a sequence of i.i.d. random
variables with distribution $\P(\eta_i = 1) = \P(\eta_i = -1) 
= 1/2$,
\item $(E_i )_{i \geq 1}$ is a sequence of exponential distributed i.i.d. random
variables with parameter 1,
\item $(U_i)_{i \geq 1}$ is a sequence of i.i.d. random
variables with uniform distribution on $[0,1]$.
\end{itemize}
All these sequences are assumed to be independent. Applying  \cite[Theorem 5.3]{JR} 
to our particular case,
\begin{equation}\label{tempal}
TL(t) = \sum_{i\geq 1}  1_{(0,t]}(\X_i)\, \bigg(\left(\frac{\alpha \Y_i}{t_1}   \right)^{-1/\alpha} \wedge E_i U_i^{1/\alpha}\bigg) \, \eta_i, 
\end{equation}
where $x \wedge y$ denotes the minimum of $x$ and $y$, is a symmetric 
tempered stable motion on $[0,t_1]$.
Note that \eqref{tempal} essentially has the form of \eqref{sum} modified by the tempering term, observing that  $\Pi = \{(\X_i, \Y_i \eta_i/2t_1)\}_{i\geq 1}$ has the distribution of a Poisson point process on $(0,t_1) \times \mathbb{R}$ and that $\alpha^{-1/\alpha}$ is simply a normailzation constant.

The settings of Sections \ref{sec:deter} and \ref{sec:stoch} are easily 
adapted to deal with tempering. For $\alpha: \mathbb{R} \to [a,b]  \subset (0,1)$  and a sequence of isolated
points $(x_i,y_i,e_i,u_i,\varepsilon_i)_i \in 
(t_0,t_1) \times \left(\mathbb{R}^+\right)^2 \times [0,1] \times \{-1,1\}$ 
such that  $(y_i)_i$ is  increasing  with $\sum_{i\geq 1} y_i^{-1/b} < \infty$, set
\begin{equation}\label{tempidentint}
f(t) = a_0 +\sum_{i\geq 1} 1_{(t_0,t]}(x_i) 
\bigg(\left(\frac{\alpha(f(x_{i-})) y_i}{t_1-t_0}\right)^{-1/\alpha(f(x_{i-}))} 
\wedge e_i u_i^{1/\alpha(f(x_{i-}))}\bigg) \, \eta_i 
\end{equation}
for $t \in [t_0,t_1)$.
Then the series is absolutely convergent, and an inspection of the proofs of Theorems \ref{finlim} and  \ref{thmdet} reveals that the same
steps can be followed with little modification. The only notable change is
that the function $z \mapsto \Big(\left(\frac{\alpha(z)) y_i}{t_1-t_0}\right)^{-1/\alpha(z))}\wedge e_i u_i^{1/\alpha(z))}\Big)$ is not necessarily differentiable.
However, as the pointwise minimum of two Lipschitz functions, it is again
Lipschitz with Lipshitz constant the maximum of the norm of the derivatives
of the two functions involved, which is all we need. Thus there
exists a unique function $f$ satisfying \eqref{tempidentint}.

Moving to the stochastic case, and choosing $(\Y_i)_{i \geq 1}, 
(\X_i)_{i \geq 1}, (\eta_i)_{i \geq 1}, (E_i)_{i \geq 1}$ and $(U_i)_{i \geq 1}$
as above, we find there exists a Markov process satisfying:
$$TZ(t) = \sum_{i\geq 1}  1_{(0,t]}(\X_i)\, \bigg(\left(\frac{\alpha(Z(\X_{i-})) \Y_i}{T}   \right)^{-1/\alpha(Z(\X_{i-}))} \wedge E_i U_i^{1/\alpha(Z(X_{i-}))}\bigg) \, \eta_i.$$
The proof of Proposition \ref{cty} adapts without any modification, so that
$TZ$ is also almost surely right-$\beta$-H\"{o}lder continuous for all $0< \beta<1/\alpha(Z(t))$.

\section{Simulation}
\setcounter{equation}{0}
\setcounter{theo}{0}

\subsection{Difficulties with simulation}

Simulation of these processes $Z$ is fraught with difficulties.
Methods for simulating paths of random process are usually based
on the joint probability distribution function or the joint
characteristic function. Such probabilistic properties of $Z$ are not 
known at this time. This strongly
restricts the tools available for simulation and
the only method that can be used at this stage is based on the Poisson
point representations \eqref{identrand} and \eqref{zedn}

In general, using series representations to simulate stable random
variables is not considered a practical method because convergence 
is rather slow \cite[p. 26]{Bk_Sam}. A further complication arises in our case, even for calculating deterministic jump functions defined by summation over point sets described in Section 2. One might hope that the error in approximating $f$ in \eqref{ident} by $f_n$ in \eqref{identfin} obtained by restricting the sum to $(x,y) \in \Pi$ with $|y|\leq n$ would be of the order
$$\Big| \sum_{ (x,y) \in \Pi \, :\,  |y|> n}  y^{\langle-1/\alpha(f_n(x_-))\rangle}\Big|
\leq  \sum_{ (x,y) \in \Pi \, :\,  |y|> n} |y|^{-1/b}.$$
Nevertheless, this need not be the case.  For   $m>n$ the differences between $f_m$ and $f_n$ are not just due to the additional summands in $f_m$. If $(x,y)\in \Pi$ with $|y|\leq n $ and there is some $(x',y')\in \Pi$ with $x' <x$ and $n<|y'|\leq m$, then $\alpha(f_n(x_-))$ and $\alpha(f_m(x_-))$ are likely to differ, in which case the terms $y^{\langle-1/\alpha(f_n(x_-))\rangle}$ and $y^{\langle-1/\alpha(f_m(x_-))\rangle}$ may differ enormously if $y$ is small. Thus points $(x,y)\in \Pi$ with $y$ small can lead to unexpectedly large changes when $n$ is incremented in \eqref{identfin}.  Thus whilst  $f_n$ converges uniformly to $f$, convergence may be much slower than one might hope.
Clearly this phenomenum will also be present in the random function $Z$ in \eqref{identrand}. The calculation of Corollary  \ref{estep} and estimate of Theorem \ref{thmrand3} provides some control of this effect. 

A further complication of any simulation based on Poisson sums is that error estimates will inevitably depend on the rate of convergence of the sums, which will depend on the particular realisation of the point distribution. For a Poisson point process $\Pi$ on  $(t_0,t_1)\times \mathbb{R}$ with Lebesgue measure ${\mathcal L}^2$ as mean measure, $\# \{(\X,\Y) \in \Pi : |Y| \leq n\} \leq C n$ for all $n$ for some random $C< \infty $ almost surely. However, $C$ cannot be determined by sampling any bounded set of points. Thus the best that can be hoped for is to simulate an approximation depending on a bounded set of $(\X,\Y)$ in such a way that there is a high probability that this will differ from the random function by at most a prescribed small amount.
We will show below how to find a value of $n$ required to ensure  that $\|Z_n-Z\|_\infty <\epsilon $ with prescribed probability. However, such an $n$ will be extremely large for practical values. 
At this time we recognize that no practical method exists to simulate quickly these processes with high precision.

\subsection{An approach to simulation}\label{sec5.2}
We first assume that we do not allow points  $(\X,\Y)\in \Pi$ such that $|\Y|$ is too small. Thus, as in Theorem \ref{thmrand3}, we make the assumption that $|\Y|\geq K$ for some $K>0$; in other words we run the Poisson process on 
$(t_0,t_1)\times (-\infty, -K] \cup [K,\infty)$. We can then apply Markov's inequality to  \eqref{diffbound1} to estimate the value of $n$ required. For convenience we take $[t_0,t_1] = [0,T]$.
In the setting we propose the following procedure to obtain a number $N(\epsilon)$ such that 
\be\label{nep}
\P(\|Z_{N(\epsilon)}- Z\|_\infty <\epsilon) \ > \ 1-\epsilon
\ee
where $Z_n$ is given by \eqref{zedn}, so that $Z$ may be simulated by the approximation $Z_{N(\epsilon)}$.
\medskip

$\bullet$  Given $\alpha: \mathbb{R} \to (0,1)$ find the optimal $a,b$ such that $\alpha(\xi)\in [a,b]$ for all $\xi$. Let $M = \sup_{\xi \in \mathbb{R}} |\alpha'(\xi)|/\alpha(\xi)^2$ (it may be enough just to consider $\xi$ ranging over a subinterval of $\mathbb{R}$ here).
\medskip

$\bullet$  Provided $K\geq 1$, \eqref{diffbound1} implies that 
\begin{eqnarray}
\E\big(\|Z_n -Z\|_\infty\big) &\leq &\frac{2bT}{1-b} \exp \bigg(2TM\int_K^\infty y^{-1/b}\log y\, dy\bigg) n^{-(1-b)/b}\nonumber\\
&= & \frac{2bT}{1-b} \exp \bigg(2TM\Big( \frac{b}{1-b}\log K  +\Big(\frac{b}{1-b}\Big)^2 K^{-(1-b)/b}      \Big)\bigg)\, n^{-(1-b)/b}.\qquad \label{bound1}
\end{eqnarray} 
Use this estimate to choose $n = N(\epsilon)$ so that  $\E\big(\|Z_n -Z\|_\infty\big)< \epsilon^2$. Then , 
$$\P(\|Z_{N(\epsilon)}- Z\|_\infty \geq \epsilon)\ <\ \E\big(\|Z_{N(\epsilon)} -Z\|_\infty\big)/\epsilon \ <\ \epsilon.$$

$\bullet$  Now find a realisation of a Poisson point process $\Pi$ with ${\mathcal L}^2$ as mean measure on $(0,T)\times (-N(\epsilon), -K] \cup [K,N(\epsilon))$. To do this we may make use of the following well-known property of the 
Poisson process \cite[p.62]{Dig},\cite{LS}: the $y$-coordinates of points $(\X,\Y)\in \Pi$ in the
semi-infinite strip $(0,T) \times\mathbb{R}^+$ form a one-dimensional Poisson
process with intensity $T$. In particular, the differences between successive
increasing $\Y$  are independent realisations of an exponentially
distributed random variable with distribution function $F(v) = 1 - \exp\left(-Tv \right), \ v>0$. Thus starting at $K$ and incrementing by these exponential random variables until we get a value with $\Y > N(\epsilon)$ and taking the corresponding $\X$   independently and uniformly distributed on $(0,T)$, we get  a realisation of $\Pi$ on  $(0,T)\times  [K,N(\epsilon))$. Similarly we get a realisation of $\Pi$ on $(0,T)\times (-N(\epsilon), -K] $.
\medskip

$\bullet$  To simulate $Z_n$ given this  Poisson point process we discretise
$[0,T]$ in an uniform way so that the time step is smaller than
$\min_{1\leq K \leq N-1} \left(\X^{(K+1)} - \X^{(K)}\right)$, where $(\X^{(K)})_K$ is
the ordered sequence of $\X$ values in $\Pi$  on $(0,T)\times (-N(\epsilon), -K] \cup [K,N(\epsilon))$. This is to ensure that at most one jump may occur between successive points at which the approximating process is estimated.
Let $0=t_0 < t_1 < \ldots < t_L=T$ be the times at which we will estimate 
$Z_n$. Starting with $Z_n(0)=0$, we let $Z_n(t_k)=0$ for all $k$ such that
$t_k<\X^{(1)}$. We then set 
$Z_n(t_{k+1})={\Y^{(1)}}^{\langle-1/\alpha(Z_n(t_k))\rangle}
={\Y^{(1)}}^{\langle-1/\alpha(0)\rangle}$. We iterate this procedure until the
terminal time is reached.
\medskip

$\bullet$  If we remove the assumption that $|\Y| \geq K$ for all $(\X,\Y)$, so that $\Pi$ becomes a Poisson point process over $(0,T) \times \mathbb{R}$, we can still get an estimate for $N(\epsilon)$ such that \eqref{nep} is satisfied, but it is likely to be much larger. Since $\#\{(\X,\Y) \in \Pi :  |\Y| \leq K\}$ has  a Poisson distribution with mean $2KT$, setting $K = \log (1/(1-\epsilon)) \big/2T$ gives 
$\P (\#\{(\X,\Y) \in \Pi :  |\Y| \leq K\} = 0) <\epsilon$. Thus, proceeding as above with this  $K$, we obtain a value of $N(\epsilon)$ such that 
$$ \P(\|Z_{N(\epsilon)}- Z\|_\infty <\epsilon) \ > \ 1-2\epsilon$$
in place of \eqref{nep}.
\medskip

Whilst for certain parameters, using \eqref{bound1} will give an impossibly large estimate for $N(\epsilon)$, in other cases the values given are not unusuable. For example, taking $K=1$, and $\alpha (t)$ with $0<b<0.5$ gives
$N(\epsilon) \geq  2T(\exp(2MT))/\epsilon^2$.

\subsection{Examples}

Examples of simulated self-stabilizing processes are displayed on Figures
\ref{fig:selfstable1} and \ref{fig:selfstable2}. One remarks that the intensity
of jumps is indeed governed by the value of the process through the $\alpha$
function. In addition, the local roughness of the paths seems in visual agreement with
Proposition \ref{cty}.

\begin{figure}[htbp]
\centering
\includegraphics[width=0.8\textwidth,height=0.5\textwidth]{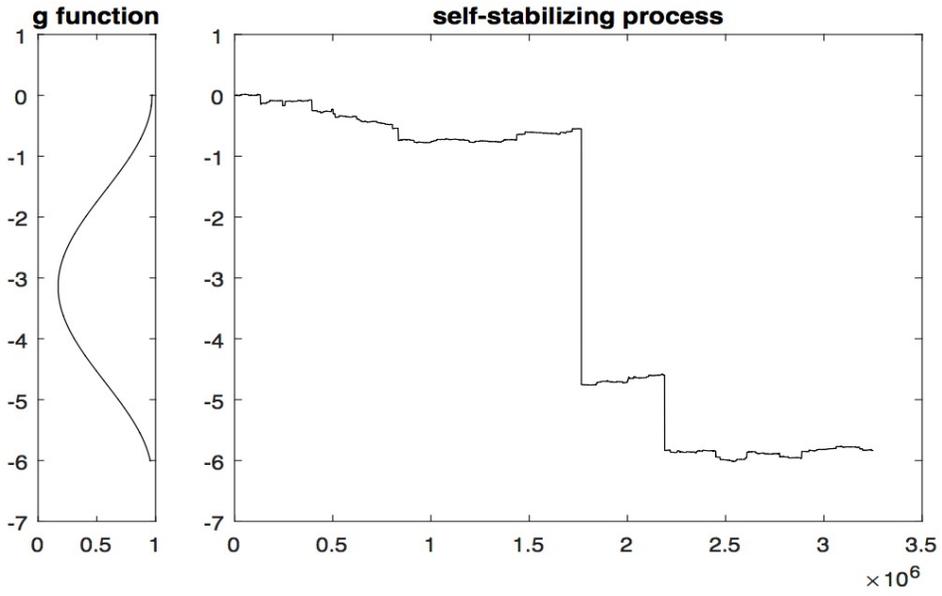}
\caption{left : self-stabilizing function $\alpha(z) = 0.57+0.4 \cos(z)$.
Right: corresponding realization of a self-stabilizing process. \label{fig:selfstable1}}
\end{figure}

\begin{figure}[htbp]
\centering
\includegraphics[width=0.8\textwidth,height=0.5\textwidth]{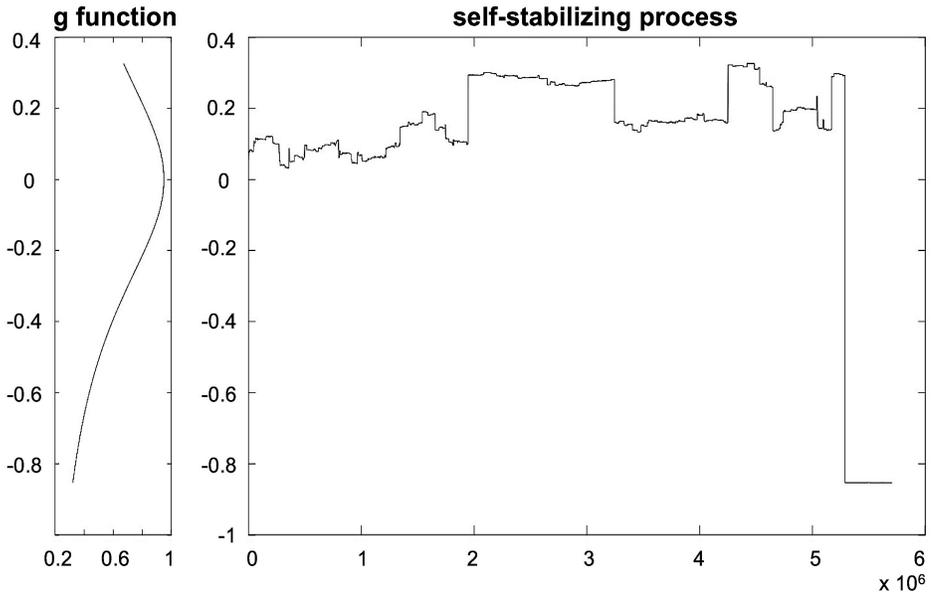}
\caption{left : self-stabilizing function $\alpha(z) = 0.15+\frac{0.8}{1+5z^2}$.
Right: corresponding realization of a self-stabilizing process. \label{fig:selfstable2}}
\end{figure}

\section*{Acknowledgements}
The authors thank the two referees for their careful reading of the paper and helpful comments. KJF gratefully acknowledges the hospitality of Institut Mittag-Leffler in  Sweden, where part of this work was carried out. JLV is grateful to SMABTP for financial support.

\bibliographystyle{plain}

\bigskip
\end{document}